\documentclass[a4paper,11pt]{article}
\usepackage[makeroom]{cancel}
\usepackage[table]{xcolor}
\usepackage{amsmath}
\usepackage{amsfonts}
\usepackage{geometry}
\usepackage[utf8]{inputenc}
\usepackage{enumerate} 
\usepackage{subfig}
\usepackage{tikz}
\usepackage{tikz-3dplot}
\usepackage{color}
\usepackage{float}
\usepackage{leftidx}
\usepackage{bigints}
\usepackage[mathscr]{euscript}
\usepackage{amssymb,amsmath,dsfont,url,epsfig}
\usepackage{titlesec, blindtext, color}
\usepackage{graphicx, psfrag, epsf, times}
\definecolor{gray75}{gray}{0.75}

\newcommand{\sln}{\linespread{1}}
\newcommand*{\email}[1]{\href{mailto:#1}{\nolinkurl{#1}} } 
\titleformat{\chapter}[block]{\LARGE\bfseries\sln}{Chapter \thechapter}{11pt}{\newline\huge\bfseries}
\newtheorem{thm}{Theorem}[section]
\newtheorem{rem}{Remark}[section]
\newtheorem{defn}{Definition}[section]

\newenvironment{proof}{\paragraph{Proof:}}{\hfill$\square$}
\newtheorem{lem}{Lemma}[section]
\newtheorem{proposition}{Proposition}[section]

\begin{document}
	\title{The Funk–Finsler Structure in the Constant Curvature Spaces}
	\author{Ashok Kumar\footnote{E-mail: ashokkumar@hri.res.in; Harish-Chandra Research Institute, A CI of Homi Bhabha National Institute, Chhatnag Road, Jhunsi, Prayagraj-211019, India.}, Hemangi Madhusudan Shah\footnote{E-mail: hemangimshah@hri.res.in ; Harish-Chandra Research Institute, A CI of Homi Bhabha National Institute, Chhatnag Road, Jhunsi, Prayagraj-211019, India.} and Bankteshwar Tiwari\footnote{E-mail: btiwari@bhu.ac.in; Centre for Interdisciplinary Mathematical Sciences, Banaras Hindu University, Varanasi-221005, India}}
	\date{} \maketitle
	
	\begin{abstract}
		\noindent 
		In this paper, we {\it find} the infinitesimal structure of Funk-Finsler metric in spaces of constant curvature. We investigate the geometry of this Funk-Finsler metric by explicitly computing its $S$-curvature, Riemann curvature, Ricci curvature, and flag curvature. Moreover, we show that the $S$-curvature of the Funk-Finsler metric in hyperbolic space  is bounded above by $\frac{3}{2}$, in spherical space bounded below by $\frac{3}{2}$, and in Euclidean case it is identically equal to $\frac{3}{2}$. Further, we show that the flag curvature of the Funk-Finsler metric in hyperbolic space  is bounded above by $-\frac{1}{4}$, in spherical space bounded below by $-\frac{1}{4}$, and in Euclidean case it is identically equal to $-\frac{1}{4}$.
		
	\end{abstract}
	\section{Introduction}\label{sec1}
	The Funk metrics on a convex subset in Euclidean space was introduced as an example to support Hilbert's fourth problem of 1900 ICM: {\it Find all non-Euclidean metrics having line segments as their geodesics}. The infinitesimal structure of the Funk metric on the unit ball in the Euclidean space is a well-known Finsler metric of constant flag curvature $-\frac{1}{4}$. The Funk metric on the unit disc is also an example of the Randers metric on the unit ball.\\
	The Funk distance  $d_{F, \Omega}(x, y)$ between any two points $x$ and $y$ in a bounded convex set $\Omega$ in the Euclidean space $\mathbb{R}^n$ is given by
	\begin{equation*}
		d_{F, \Omega}(x, y):=
		\left\{
		\begin{array}{ll}
			\log\left( \frac{|x-\mathfrak{a}|}{|y-\mathfrak{a}|} \right), & \mbox{if }  x \neq y \\
			0, & \mbox{if } x = y. 
		\end{array}
		\right.
	\end{equation*} 
	where $|.|$ represents the Euclidean norm in $\mathbb{R}^n$, $\mathfrak{a}=\overrightarrow{xy} \cap \partial \Omega$ and $\overrightarrow{xy}$  denotes the geodesic ray starting from $x$ and passing through $y$  and  $\partial \Omega=\bar{\Omega} \setminus \Omega$ denotes the  boundary of $\Omega$ \cite[Chapter 2, $\S 2$]{HHG}.\\ \newpage
	\noindent Also, the corresponding  Finsler structure on the convex set $\Omega$  is given by
	\begin{equation*}
		F_\Omega(x, \xi) := \inf  \left\lbrace t > 0 ~\vert  \left( x+\frac{\xi}{t} \right) \in   \Omega \right\rbrace,
	\end{equation*}
	where $\xi \in T_x\Omega$ \cite[Chapter 3, $\S 3$]{HHG}.
	
	\noindent Papadopoulos and Yamada \cite[$\S 2$]{PAYS} define the variational Funk metric $d_{F,\Omega}$ in an open-bounded convex subset $\Omega \subset \mathbb{R}^n$ equipped with Euclidean metric as
	\begin{equation*}
		d_{F,\Omega}(x,y)=\sup \limits_{\pi \in \mathscr{P} } \log \frac{ d(x, \pi)}{ d(y, \pi)},
	\end{equation*}
	where $x,y  \in \Omega$; $\mathscr{P}$ is the set of all supporting hyperplanes of $\Omega$ and $d(.,\pi)$ is the Euclidean distance function defined in $\Omega$ from the supporting hyperplane $\pi$. It is well known that the sine formula for $\Delta ABC$  in the Euclidean geometry is given by $$\frac{\sin A}{a}=\frac{\sin B}{b}=\frac{\sin C}{c};$$ while the sine formula in the hyperbolic geometry and spherical geometry are given by  $$\frac{\sinh A}{\sinh a}=\frac{\sinh B}{\sinh b}=\frac{\sinh C}{\sinh c}\ \  \text{and} \ \
	\frac{\sin A}{\sin a}=\frac{\sin B}{\sin b}=\frac{\sin C}{\sin c}$$
	respectively; here $a,b,c$ in these formulae are the Euclidean/ hyperbolic/ spherical lengths of sides $BC,CA$ and $AB$ of geodesic triangle $ ABC$ in the Euclidean/ hyperbolic/ spherical geometry respectively.
	
	Therefore, while studying the similarity of triangles in hyperbolic and spherical geometry, only the ratio of the lengths of its sides does not appear; in fact, it is weighted by $\sinh$ and $\sin$ functions respectively.
	Considering these facts, Papadopoulos and Yamada define the Funk distance  in hyperbolic and spherical geometry as follows, which we now call the Papadopoulos-Yamada-Funk distance:\\
	The variational formula for the Papadopoulos-Yamada-Funk distance $d_{F_H,\Omega}$, on an open bounded convex set $\Omega $ of the Hyperbolic space $\mathbb{H}^n$, is given by  (see \cite[$\S 3$]{PAYS}):
	\begin{equation}\label{int1}
		d_{F_H,\Omega}(x,y)=\sup \limits_{\pi \in \mathscr{P}} \log 
		\frac{\sinh d_H(x, \pi)}{\sinh d_H(y, \pi)}.
	\end{equation}
	where $x, y $ in $\Omega$; $\mathscr{P}$ is the set of all supporting hypersurfaces of $\Omega$ in $\mathbb{H}^n$ and $d_H(.,\pi)$ is the hyperbolic distance function defined in $\Omega$ from the supporting hypersurface $\pi$.
	\\ 
	Also, the variational formula for the Papadopoulos-Yamada-Funk distance $d_{F_S,\Omega}$ on a convex set $\Omega $ of the spherical space $\mathbb{S}^n$,  such that the diameter of $\Omega$ is less or equal to $\frac{\pi}{2}$ is given by (see \cite[$\S 3$]{PAYS}):
	\begin{equation}\label{V}
		d_{F_S,\Omega}(x,y)=\sup \limits_{\pi \in \mathscr{P} } \log \frac{\sin d_s(x, \pi)}{\sin d_s(y, \pi)}.
	\end{equation}
	where $x, y $ in $\Omega$; $\mathscr{P}$ is the set of all supporting hypersurfaces of $\Omega$ in $\mathbb{S}^n$ and $d_s(.,\pi)$ is the spherical distance function defined in $\Omega$ from the supporting hypersurface $\pi$. \\
	More precisely, if $\Omega$ is a strictly convex set on the sphere with diameter less than or equal to $\frac{\pi}{2}$, then the equation \eqref{V} reduces to
	\begin{equation}
		d_{F_S,\Omega}(x,y)= \log \frac{\sin d_s(x, \mathfrak{a})}{\sin d_s(y, \mathfrak{a})}.
	\end{equation}
	where $\mathfrak{a}=\overrightarrow{xy} \cap \partial \Omega$ and $\overrightarrow{xy}$  denotes the geodesic ray starting from $x$ and passing through $y$.\\
	
	\noindent In the sequel, we denote by $|.|$ and $\langle . \rangle$,  the Euclidean norm and the Euclidean inner product, respectively and $\mathbb{D}_E( r)=\left\lbrace (x^1,x^2) \in \mathbb{R}^2 :  (x^1)^2+(x^2)^2 <  r^2 \right\rbrace$, the Euclidean disc centered at the origin with radius $r$ in $\mathbb{R}^2$. \\\\
	In \cite{AHB2}, the present authors study the infinitesimal Funk-Finsler structure in various models of hyperbolic geometry. In the present paper, first we find the infinitesimal Funk-Finsler structure in spherical geometry and then we obtain a  formula for infinitesimal Funk-Finsler structure in the spaces of constant curvature $\epsilon=\{0,1,-1\}$. More precisely, we have the following theorems:
	
			\begin{thm}\label{thm1.1}
				The Funk-Finsler metric $\mathcal{F}_{\epsilon}$ on the disc $\mathbb{D}_E(r)$, centered at origin and radius $0<r\leq r_{\epsilon}$, equipped with a Riemannian metric of constant curvature $\epsilon \in \{-1,0,1\}$ is a Randers metric given by
				$\displaystyle \mathcal{F}_{\epsilon}(x,\xi )=\alpha_F (x,\xi )+\beta_F(x,\xi )$,
				where $\displaystyle \alpha (x,\xi )=\frac{\sqrt{\left(r^2-|x|^2 \right) |\xi|^2+ \langle x, \xi \rangle ^2 }}{r^2-|x|^2}$ is a Riemannian metric in the disc and $\beta$ is a closed $1$-form given by $\displaystyle \beta(x,\xi )=\frac{(1+\epsilon r^2) \langle x,\xi \rangle}{(r^2-|x|^2)(1+\epsilon|x|^2)}$ for all $x \in \mathbb{D}_E(r),\ \xi \in T_x\mathbb{D}_E(r)$, where $r_{\epsilon}=1 $ for $\epsilon= \pm 1 $ and $r_{0}<\infty$.
			\end{thm}
			
			\begin{thm}\label{thm1.3}
				The $S$-curvature of the Funk-Finsler metric $\mathcal{F}_{\epsilon}$ in the hyperbolic space is bounded above by $\frac{3}{2}$, in the spherical case it is bounded below by $\frac{3}{2} $, and in the Euclidean case it is identically equal to $\frac{3}{2} $.
			\end{thm}
			
			\begin{thm}\label{thm1.2}
				The flag curvature of the Funk-Finsler metric $\mathcal{F}_{\epsilon}$ in the hyperbolic space is bounded above by $-\frac{1}{4}$, in the spherical case bounded below by $-\frac{1}{4}$, and in Euclidean case it is identically equal to $-\frac{1}{4}$.
			\end{thm}
			
			\noindent In view of Theorem \ref{thm1.2}, the following remark is in order:
			\begin{rem}
				Since the Funk-Finsler metric $\mathcal{F}_{\epsilon}$ in the hyperbolic space is bounded above by a negative constant $-\frac{1}{4}$ and hence the disc $\mathbb{D}_E(r)$ equipped with this  Funk-Finsler metric $\mathcal{F}_{\epsilon}$ is a Cartan-Hadamard manifold. 
			
			\end{rem}
			\noindent The paper is organized as follows.
			In Sect.\ref{sec2}, we discuss the preliminaries required for the paper.  In Sect.\ref{sec3}, we explicitly construct the Funk-Finsler metric on a spherical disc and give the Proof of Theorem \ref{thm1.1}. In Sect.\ref{sec4}, we investigate the geometry of the Funk-Finsler metric on the disc $\mathbb{D}_E(r)$ and compute explicitly the spray coefficients, $S$-curvature,  Riemann curvature, Ricci curvature, the flag curvature, and give the proof of Theorem \ref{thm1.3} and Theorem \ref{thm1.2}. Further, in Sect.\ref{sec5}, we obtain the Zermelo Navigation data for the Funk-Finsler metric on the disc $\mathbb{D}_E(r)$. Finally, in Sect.\ref{sec6}, which is the Appendix, we find a formula for distance between two points on the sphere, which is useful in computing the Funk-Finsler structure on the sphere. 
			\section{Preliminaries}\label{sec2}
			The theory of Finsler manifolds can be considered as a generalization of Riemannian manifolds, where the Riemannian metric is replaced by a so-called {\it Finsler} metric. A Finsler metric is a smoothly varying family of Minkowski norms in each tangent space of the manifold.\\
			Let $ M $ be an $n$-dimensional smooth manifold and let $T_{x}M$ denote the tangent space of $M$ at $x$. The tangent bundle $TM$  of $M$ is the disjoint union of tangent spaces: $TM:= \sqcup _{x \in M}T_xM $. We denote the elements of $TM$ by $(x,\xi)$, where $\xi \in T_{x}M $ and $TM_0:=TM \setminus\left\lbrace 0\right\rbrace $.
			\begin{defn}[{Finsler structure \cite[\boldmath$\S 1.2$]{SSZ};  
					\cite[\boldmath$\S 16.2$]{Shiohama-BT}  }]\label{def 3.A01} 
				A {\it Finsler structure} on the manifold $M$ is a function $F:TM \to [0,\infty)$ satisfying the following conditions: 
				\begin{enumerate}[(i)]
					\item  $F$ is smooth on $TM_{0}$,
					\item  $F$ is a positively 1-homogeneous on the fibers of the tangent bundle $TM$,
					\item  The Hessian of $\displaystyle \frac{F^2}{2}$ with elements $\displaystyle g_{ij}=\frac{1}{2}\frac{\partial ^2F^2}{\partial \xi^i \partial \xi^j}$ is positive definite on $TM_0$.\\
					The pair $(M, F)$ is called a Finsler space, and $g_{ij}$ is called the fundamental tensor of the Finsler structure $F$.
				\end{enumerate}
			\end{defn}
			
			\noindent
			It is easy to see that Riemannian metrics are examples of Finsler metrics. The Randers metric is the simplest non-Riemannian example of a Finsler metric.
			\begin{defn}[{Randers Metric \cite[\boldmath$\S 1.2$]{SSZ}}]  Let $\displaystyle \alpha=\sqrt{a_{ij}(x)dx^idx^j}$ be a Riemannian metric  and $\beta=b_i(x)dx^i$ be a $1$-form on a smooth manifold  $M$ with $||\beta||_{\alpha}<1$, where $\displaystyle ||\beta||_{\alpha} = \sqrt {a^{ij}(x)b_{i}(x)b_{j}(x)}$; \  then $F(x,\xi)=\alpha(x,\xi)+\beta(x,\xi)$, for all $x \in M, \xi \in T_xM$, is called a Randers metric on $M$. 
			\end{defn}

							\begin{defn} [\textbf{The Riemann curvature tensor \cite[$\S 4.1$]{CXSZ}}]\label{def2}  
								The Riemann curvature tensor $R={{R}_\xi:T_xM \rightarrow T_xM}$, for a Finsler space  $(M^n,F)$ is defined by
								${R}_{\xi}(u)={R}^{i}_{k}(x,\xi)u^{k} \frac{\partial}{\partial x^i}$,\ $u=u^k\frac{\partial}{\partial x^k}$,\ where ${R}^{i}_{k}={R}^{i}_{k}(x,\xi)$ denote the coefficients of the Riemann curvature tensor of the Finsler metric $F$ and are given by,
								\begin{equation}\label{eqn2.2.10}
									{R}^{i}_{k}=2\frac{\partial{G}^i}{\partial	x^k}-\xi^j\frac{\partial^2{G}^i}{\partial x^j\partial	\xi^k}+2{G}^j\frac{\partial^2 {G}^i}{\partial \xi^j\partial	\xi^k}-\frac{\partial{G}^i}{\partial \xi^j}\frac{\partial {G}^j}{\partial \xi^k}.
								\end{equation}
								Here $G^i=G^i(x,\xi)$ are local functions on $TM$, called the spray coefficients defined by 
								\begin{equation}\label{eqn2.1.14}
									{G}^i=\frac{1}{4}{g}^{i{\ell}}\left\{\left[{F}^{2}\right]_{x^k\xi^{\ell}}\xi^k-	\left[{F}^{2}\right]_{x^{\ell}}\right\}.
								\end{equation} 
							\end{defn}
							\noindent	The flag curvature $K=K(x,\xi, P)$, generalizes the sectional curvature in Riemannian geometry to the Finsler geometry and does not depend on whether one is using the Berwald, the Chern or the Cartan connection.

							\begin{defn}[\textbf{Flag curvature \cite[$\S 4.1$]{CXSZ}}] 
								\textnormal{	For a  plane $P\subset T_xM$ containing a non-zero vector $\xi$ called \textit{pole}, the \textit{flag curvature} $\textbf{K}(x,\xi,P)$ is defined by
									\begin{equation}\label{eqn2.1.11}
										\textbf{K}(x,\xi,P) :=\frac{g_\xi(R_\xi(u), u)}{g_\xi(\xi, \xi)g_\xi(u, u)- g_\xi(\xi, u)^2},
									\end{equation}
									where $u\in P$ is such that $P=\text{span} \left\{ \xi,u\right\} $.\\
									If $\textbf{K}(x,\xi,P)=\textbf{K}(x,\xi)$, then the	Finsler metric $F$ is said to be of scalar flag curvature and if $\textbf{K}(x,\xi,P)=\text{constant}$, then the Finsler metric $F$ is said to be of constant flag curvature .}
							\end{defn}
							The relation between the Riemann curvature  $R^i_j$ and the scalar flag curvature $\textbf{K}(x,\xi)$ of a Finsler metric $F$ is given by (see for more detail \cite[$\S 4.1$]{CXSZ})
							\begin{equation}\label{eqn2.1.12}
								R^i_j= \textbf{K}(x,\xi)\left\lbrace F^2 \delta^i_j-FF_{\xi^j}\xi^i\right\rbrace.
							\end{equation}
							
							\noindent It is well known that there is no canonical volume form on a Finsler manifold, like in the Riemannian case. Indeed, there are several well-known volume forms on the Finsler manifold, for instance,  the {\it Busemann-Hausdorff} volume form, the {\it Holmes-Thompson} volume form,  the {\it maximum } volume form,  the {\it minimum } volume form, etc. Here we discuss only the Busemann-Hausdorff volume form.
							\begin{defn}[{Busemann-Hausdorff Volume form in Finsler manifolds \cite[\boldmath$\S 2.2$]{SZ}; \cite[\boldmath$\S 2.7$]{BT}}]
								Let $(M,F)$ be an $n$-dimensional Finsler manifold and $(U,x^i)$ be a coordinate chart containing the point $x$. Let $\{\frac{\partial}{\partial x^i}\big |_x\}_{i=1}^n$ be the basis of $T_xM$ induced from the coordinate chart $(U,x^i)$. Then the \textit{Busemann-Hausdorff} volume form on the Finsler manifold $(M,F)$ is defined as:
								$dV_{BH}=\sigma_{BH}(x) \ dx$, where
								\begin{equation}
									\sigma_{BH}(x)=\frac{\emph{Vol} ( B^n(1))}{\textnormal{Vol} \left\lbrace (\xi^i)\in \mathbb R^n: F(x,\xi^i\frac{\partial}{\partial x^i}\big |_x)<1\right\rbrace},
								\end{equation}
								and $dx=dx^1\wedge dx^2\wedge \dots \wedge dx^n$.
								Here $ B^n(1)$ denotes the Euclidean unit ball and \textnormal{Vol} denotes the canonical volume.
							\end{defn}
							\noindent
							The Busemann-Hausdorff volume form of the Randers metric can be explicitly given as follows:
							\begin{lem}[{\cite[\boldmath$\S 3$]{W}}]\label{lem 3.A1}
								The Busemann-Hausdorff volume form of the Randers metric $F =\alpha + \beta$ is given by,
								\begin{equation}\label{eqn 3.A38}
									dV_{BH} =\sigma_{BH}(x)dV_\alpha= \left( 1-||\beta||^2_\alpha\right)^{\frac{n+1}{2}} dV_\alpha,
								\end{equation}
								where $dV_\alpha=\sqrt{\det (a_{ij})} \ dx$.
							\end{lem}
							\noindent For the \textit{Busemann-Hausdorff} volume form 
							$dV_{BH}=\sigma_{BH}(x)dx$ on Finsler manifold $(M,F)$, the \textit{distortion} $\tau$ is defined by (see, \cite[$\S 5.1$]{SSZ})
							\begin{equation*}
								\tau (x,\xi) :=\ln \frac{\sqrt{\det(g_{ij}(x,\xi))}}{\sigma _{BH}(x)}.
							\end{equation*}
							Now we define $S$-curvature of the Finsler manifold $(M, F)$ with respect to the volume form $dV_{BH}$.
							
							\begin{defn}[{S-curvature, \cite[\boldmath$\S 5.1$]{SSZ}}]
								\textnormal{	For a vector $\xi\in T_xM\backslash \left\lbrace 0\right\rbrace$, let $\gamma=\gamma(t)$ be the geodesic with $\gamma(0)=x$ and $\dot{\gamma}(0)=\xi$.
									Then the $S$-curvature of the Finsler metric $F$ is defined by 	
									\begin{equation*}
										S(x, \xi)=\frac{d}{dt}\left[ \tau \left(\gamma(t), \dot{\gamma}(t)\right) \right]|_{t=0}.
								\end{equation*}}
							\end{defn}
							The $S$-curvature of $F$ in  terms of spray coefficients is given by
							\begin{equation}\label{eqn2.1.15}
								S(x,\xi)=\frac{\partial G^m}{\partial \xi^m}-\xi^m\frac{\partial\left( \ln\sigma_{BH}\right) }{\partial x^m},
							\end{equation} 
							where $G^m$ are given by  \eqref{eqn2.1.14}.
							\begin{defn}[{Berwald space, \cite[\boldmath$\S 2.1$]{SSZ}}]
								A Finsler metric $F$ on a manifold $M$ is called a Berwald
								metric if in any standard local coordinate system $(x^i,\xi^i)$  in $TM_0:=TM \setminus\left\lbrace 0\right\rbrace$, the Christoffel symbols $\Gamma^i_{jk}=\Gamma^i_{jk}(x)$ are functions of $x \in M$ only, in which case, $G^i=\frac{1}{2}\Gamma^i_{jk}(x) \xi^j \xi^k$ are quadratic in $\xi=\xi^i \frac{\partial}{\partial x^i}|_x$.
							\end{defn}
							\begin{proposition}\cite[Lemma $3.1.2$]{SSZ}\label{ppn2.1}
								The following three conditions are equivalent for a Randers metric $F=\alpha +\beta$:
								\begin{itemize}
									\item [(a)] $F$ is a Landsberg metric;
									\item [(b)]  $F$  is a Berwald metric;
									\item [(c)] $\beta$ is parallel with respect to $\alpha$.
								\end{itemize}
							\end{proposition}
							
							\begin{defn}[{Douglas space, \cite[\boldmath$\S 5.2$]{CXSZ}}]
								A Finsler metric is a Douglas metric if and only if $G^i\xi^j-G^j\xi^i$ are homogeneous polynomials in $(\xi^i)$ of degree three.
							\end{defn}
							
							\begin{proposition}\cite[Theorem $5.2.1$]{CXSZ}\label{ppn2.2}
								A Randers metric $F=\alpha +\beta$ is a Douglas metric if and
								only if $\beta$ is closed.
							\end{proposition}
							
							\begin{defn}[{Projectively flat space, \cite[\boldmath$\S 3.4$]{SSZ}}]\label{PF}
								A Finsler metric $F = F(x,\xi)$ on an open subset $\mathcal{U}\in \mathbb{R}^n$ is said to be \textit{projectively flat} if all geodesics are straight in $\mathcal{U}$, that is, $\sigma(t)=f(t) a+ b$  for some constant vectors $a,b \in \mathbb{R}^n$.  A Finsler metric $F$
								on a manifold $M$ is said to be locally projectively flat if at any point, there is a local coordinate system $(x^i)$ in which $F$ is projectively flat.
							\end{defn}
							
							\begin{proposition} \cite[Proposition $3.4.8$]{SSZ}\label{ppn2.3}
								A Randers metric $F=\alpha +\beta$ is locally projectively
								flat if and only if $\alpha$ is locally projectively flat and $\beta$ is closed.
							\end{proposition}
						
					\noindent  In the sequel, we restrict ourselves to the discussion on dimension $2$ only.

			\section{The Funk- Finsler metric in the constant curvature spaces}\label{sec3}
			In this section, we construct the Funk-Finsler structure in constant curvature spaces. More precisely, we give a formula of the Funk-Finsler metric in the space of constant curvature $\epsilon \in \{-1,0,1\}$.\\
			
			\noindent It is well known that the Funk-Finsler metric in Euclidean disc $\mathbb{D}_E(r)$ of radius $r$ is given by (See \cite[\S $2.1$]{SSZ})
			\begin{equation}\label{eqn0.1}
				F_{\mathbb{D}_E(r)}(x, \xi) =\frac{\sqrt{\left(r^2-|x|^2 \right) |\xi|^2+ \langle x , \xi \rangle ^2 }+ \langle x , \xi \rangle}{r^2-|x|^2},~~ \forall x \in \mathbb{D}_E(r),\ \xi \in T_x\mathbb{D}_E(r).
			\end{equation}
			In \cite[$\S 3$]{PAYS} Papadopoulos and Yamada define  the Funk and the Hilbert distance on a convex set in the constant curvature spaces. 
			In particular, the Papadopoulos-Yamada-Funk distance  $d_{H, \Omega}(x, y)$ between any two points $x$ and $y$ in a convex set $\Omega$ in hyperbolic space is given by \eqref{int1}.
			
			\noindent Recently, the present authors have computed the Funk-Finsler metric on the Klein unit disc $\mathbb{D}_K(1)$, which is a Randers metric and given by   (See \cite[Theorem $3.1$]{AHB2})
			\begin{thm}\label{eqn0.2}(\cite[$\S 3$]{AHB2})
				The Funk-Finsler structure on the Klein unit disc $\mathbb{D}_K(1)$ is a Randers metric, given by  $\mathcal{F} =\alpha_F +\beta_F, $ where   $\displaystyle \alpha_F (x, \xi) =\frac{\sqrt{\left(r^2-|x|^2 \right) |\xi|^2+ \langle x, \xi \rangle ^2 }}{r^2-|x|^2}$ with $r = \frac{e^2-1}{e^2+1}$, is the Riemannian-Klein metric in the disc $\mathbb{D}_K(1)$ and  $\beta$ is a closed $1$-form given by
				$\displaystyle \beta_F (x, \xi) =\frac{(1-r^2) \langle x ,\xi \rangle}{(r^2-|x|^2)(1-|x|^2)}$, \ for all $x \in \mathbb{D}_K(1),\ \xi \in T_x\mathbb{D}_K(1)$  .
			\end{thm}
			\vspace{.3cm}
			\noindent
			Further, Papadopoulos and Yamada defined the Funk distance  $d_{F, \Omega}(x, y)$ between any two points $x$ and $y$ in a convex set $\Omega$  with $diam(\Omega)\leq \frac{\pi}{2}$, which we now call Papadopoulos-Yamada-Funk distance, and is given by,
			\begin{equation}\label{ath1}
				d_{F_H,\Omega}(x, y):=\log\frac{\sin d_s(x,\mathfrak{a})}{\sinh d_s(y,\mathfrak{a})},
			\end{equation}
			where  $d_s$ represents the spherical distance function, $\mathfrak{a}=\overrightarrow{xy} \cap \partial \Omega$ and $\overrightarrow{xy}$ denotes the geodesic ray starting from $x$ and passing through $y$ with in $\Omega$.

			\subsection{The Papadopoulos-Yamada-Funk  metric on spherical disc}\label{sb1}

			\noindent In this subsection, we obtain the Funk-Finsler structure in local coordinates on a strongly convex spherical domain $\Omega$ with diameter $diam(\Omega) \leq \frac{\pi}{2}$.\\
			
			\noindent As is well known, the geodesic disc of diameter less than or equal to $\pi$ is a convex set on the unit sphere. In the following, we compute the Funk-Finsler structure on a geodesic disc of diameter less than or equal to $\frac{\pi}{2}$ on a unit sphere $\mathbb{S}^2$.\\
			
			\noindent Consider the immersion $\phi : \mathbb{R}^2 \rightarrow \mathbb{S}^2_+$, given by 
			\begin{equation}\label{eqn03.1}
				\phi(x)=\left( \frac{x}{\sqrt{1+|x|^2}},\frac{1}{\sqrt{1+|x|^2}} \right),\quad x=(x^1,x^2).
			\end{equation}
			The immersion can also be thought of as the parametrization of the upper hemisphere $\mathbb{S}^2_+=\{\tilde{x}=(\tilde{x}^1,\tilde{x}^2,\tilde{x}^3) \in \mathbb{S}^2 :   \tilde{x}^3 > 0,~  \}$. The restriction of the map $\phi$ on the disc $\mathbb{D}_E(r),~0< r \leq 1$ is a parametrization of the geodesic disc $D_{\tilde{r}}(\tilde{p})$ with $\tilde{p}=(0,0,1)$ and $0 < \tilde{r} =\tan^{-1} r \leq  \frac{\pi}{4}$. In reference of this chart, we compute the Funk-Finsler structure on a geodesic disc of radius $\tilde{r} \leq \frac{\pi}{4}$.\\

			\noindent Let $\bar{g}$ be the flat metric in $\mathbb{R}^3$ given by $\bar{g}= \sqrt{(d \tilde{x}^1)^2+(d \tilde{x}^2)^2+(d \tilde{x}^3)^2}$. Then the pullback metric in $\mathbb{R}^2$ is given by
			\begin{equation}\label{eqn2.5.11811}
				g= \phi^* \bar{g}=   \sqrt{(d\phi ^1)^2+(d\phi ^2)^2+(d\phi ^3)^2}(x,\xi)=\frac{\sqrt{(1+|x|^2)|\xi|^2-\langle x, \xi \rangle ^2}}{(1+|x|^2)},
			\end{equation}
			where $(x, \xi) \in T\mathbb{R}^2$, by a simple computation, it can be shown that this metric $g= \phi^* \bar{g}$, known as spherical metric, is projectively flat, that is, its geodesics are line segments in the chart and  is of constant curvature  $+1$. 
			\begin{thm}\label{thm3.1}
				The Funk-Finsler structure for the Papadopoulos-Yamada-Funk distance on spherical geodesic disc  with diameter less than or equal to $\frac{\pi}{2}$,  is a Randers metric   given by  $\mathcal{F} =\alpha_F +\beta_F, $ where   $\displaystyle \alpha (x, \xi) =\frac{\sqrt{\left(r^2-|x|^2 \right) |\xi|^2+ \langle x, \xi \rangle ^2 }}{r^2-|x|^2}$ is a Riemannian metric and $\displaystyle \beta (x, \xi) =\frac{(1+r^2) \langle x,\xi \rangle}{(r^2-|x|^2)(1+|x|^2)}$
				is a closed $1$-form, for all $x \in \mathbb{D}_E(r),\ \xi \in T_x\mathbb{D}_E(r)$
				with $0 < r \leq 1$.

			\end{thm}
			\begin{proof}
				Let $\Omega= D_{\tilde{r}}(\tilde{p})=\left\lbrace \tilde{x} \in \mathbb{S}^2_+ :  d_s(\tilde{p},\tilde{x}) < \tilde{r} \right\rbrace$, where $\tilde{p}=(0,0,1)$ and $0< \tilde{r} \leq \frac{\pi}{4} $ be a geodesic disc with diameter $2\tilde{r} \leq \frac{\pi}{2}$. 
				In view of equation \eqref{ath1}, the Papadopoulos-Yamada-Funk distance on the geodesic disc $D_{\tilde{r}}(\tilde{p})$, which is clearly a convex set on the sphere, is given by 
				\begin{eqnarray}
					d_{F}(\tilde{x}, \tilde{y})=\log \frac{\sin d_s(\tilde{x},\tilde{\mathfrak{a}})}{\sin d_s(\tilde{y},\tilde{\mathfrak{a}})},
				\end{eqnarray}
				where $d_s(\tilde{x},\tilde{\mathfrak{a}})$ denotes the spherical distance between $\tilde{x}$ and $\tilde{\mathfrak{a}}=R_{\tilde{x},\tilde{y}} \cap \partial \Omega$. Here, $R_{\tilde{x},\tilde{y}}$ denotes the geodesic ray starting from $\tilde{x}$ passing through $\tilde{y}$.\\
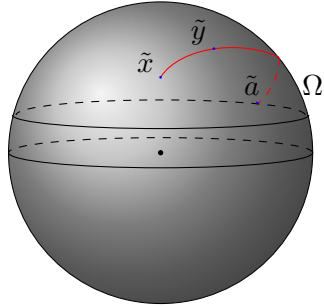
				\begin{center}
					\begin{tikzpicture}
						\shade[ball color = gray!60, opacity = 0.6] (0,0) circle (2cm);
						\draw (0,0) circle (2cm);
						\draw (-1.95,.5) arc (180:360:1.95 and 0.18);
						\draw[color=red] (0,1) arc (160:52:1 and 0.6);
						\draw[dashed,color=red] (1.56,1.26)arc  (0:-51:.8);
						\draw[dashed] (1.9,.5) arc (0:180:1.9 and 0.2);
						\draw (-2,0) arc (180:360:2 and 0.2);
						\draw[dashed] (2,0) arc (0:180:2 and 0.2);
						\fill[fill=black] (0,0) circle (1pt);
						\node at (-.2,1.2){$\tilde x$};\fill[fill=blue] (0,1) circle (.5pt);
						\node at (.5,1.6){$\tilde y$};\fill[fill=blue] (.7,1.38) circle (.5pt); \node at (2,.9){$\Omega$};
						\node at (1.2,.9){$\tilde a$};\fill[fill=blue] (1.27,.66) circle (.5pt);
						
					\end{tikzpicture}
				\end{center}
				
				\noindent Considering the chart defined in \eqref{eqn03.1}, let $x=\phi^{-1}(\tilde{x})$, $y=\phi^{-1}(\tilde{y})$ and $\mathfrak{a}=\phi^{-1}(\tilde{\mathfrak{a}})$, which is also the point of intersection of the ray segment joining $x$ through $y$ and $\partial \mathbb{D}_E(r)$ in the chart. A simple computation shows that the Papadopoulos-Yamada-Funk distance between $x$ and $y$ in the $\mathbb{D}_E(r)$ is given by (see Appendix for detail)
				\begin{eqnarray}
					d(x,y)  = \log \frac{|x-\mathfrak{a}|}{|y-\mathfrak{a}|}  +\frac{1}{2}\log \frac{1+|y|^2}{1+|x|^2}.
				\end{eqnarray}

				\noindent Let  $\gamma(t)$  be a $C^1$-curve in $\mathbb{D}_E(r)$ such that $\gamma(0)=x$ and $\dot{\gamma}(0)=\xi$, then the induced Funk-Finsler structure in $\mathbb{D}_E(r)$ is given by (see Busemann-Mayer theorem \cite[\S 6.3]{DSSZ})
				\begin{eqnarray}
					\nonumber   \mathcal{F}(x,\xi)&=&\lim  \limits_{t \rightarrow 0 } \frac{d \left( x, \gamma(t)\right)}{t}\\
					\nonumber &=& \lim  \limits_{t \rightarrow 0 } \frac{\log \frac{|x-\mathfrak{a}|}{|\gamma(t)-\mathfrak{a}|}  +\frac{1}{2}\log \frac{1+|\gamma(t)|^2}{1+|x|^2}}{t}\\
					\nonumber    &=&\frac{\sqrt{\left(r^2-|x|^2 \right) |\xi|^2+ \langle x, \xi \rangle ^2 }}{r^2-|x|^2}+ \frac{(1+ r^2) \langle x,\xi \rangle}{(r^2-|x|^2)(1+|x|^2)}\\
					\label{1}   &=& \alpha (x, \xi)+\beta (x, \xi).
				\end{eqnarray}
				
				\vspace{.5cm}
				\noindent
				If we write $\alpha (x, \xi)=\sqrt{a_{ij}\xi^i\xi^j}$, then 
				
				\begin{equation}\label{eqn2.5.118}
					(a_{ij})=\frac{1}{(r^2-|x|^2)^2}
					\begin{pmatrix}
						r^2-|x|^2+(x^1)^2& x^1x^2 \\
						x^1x^2 & r^2-|x|^2+(x^2)^2
					\end{pmatrix},
				\end{equation}
				$\det(a_{ij})=\frac{r^2}{(r^2-|x|^2)^3}$ and its inverse matrix is given by 
				\begin{equation}\label{eqn2.5.119}
					(a^{ij})=(a_{ij})^{-1}= \frac{(r^2-|x|^2)}{r^2 } 
					\begin{pmatrix}
						r^2-|x|^2+(x^2)^2& -x^1x^2 \\
						-x^1x^2 & r^2-|x|^2+(x^1)^2
					\end{pmatrix}.
				\end{equation}
				Furthermore, let $\beta(x, \xi)  =b_i(x)\xi^i$; then the coefficients $b_i(x)$ of the $1$-form $\beta_F$  is given by
				\begin{equation}\label{eqn2.5.120}
					b_i(x)=\frac{(1+r^2) x^i }{(r^2-|x|^2)(1+|x|^2)}, 
				\end{equation}
				and hence \begin{equation}\label{eqn2.5.121}
					||\beta||^2_{\alpha}=a^{ij}b_ib_j=\frac{|x|^2(1+r^2)^2}{r^2(1+|x|^2)^2}  < 1.
				\end{equation}
				\noindent It is easy to observe that $\beta=df(x)$, where $f(x)=\frac{1}{2}\log\left(\frac{1+|x|^2}{r^2-|x|^2} \right).$
				Thus, $\mathcal{F}(x, \xi)=\alpha_(x, \xi)+\beta(x, \xi)$  is a Randers metric with a closed $1$-form $\beta$.
			\end{proof}
			
			\subsection{Proof of the Theorem \ref{thm1.1}}\label{sb2}
			
			In view of  \eqref{eqn0.1}, Theorem \ref{eqn0.2} and Theorem \ref{thm3.1}, we obtain, the Funk-Finsler metric on the  disc $\mathbb{D}_E(r)$ equipped with Riemmanian metric of constant curvature $\epsilon=\{0,1,-1\}$ is given by
			\begin{equation} \label{eq 3.2.1}
				\mathcal{F}_{\epsilon}(x,\xi )=\alpha(x,\xi )+\beta(x,\xi ),
			\end{equation}
			where $$\alpha(x,\xi )=\frac{\sqrt{\left(r^2-|x|^2 \right) |\xi|^2+ \langle x, \xi \rangle ^2 }}{r^2-|x|^2} \ \ \text{and} \ \ 
			\beta(x,\xi )=\frac{(1+\epsilon r^2) \langle x,\xi \rangle}{(r^2-|x|^2)(1+\epsilon|x|^2)},$$
			for all $(x,\xi)\in T\mathbb{D}_E(r)$. 
			
			\noindent Further, we write $\beta(x, \xi)  =b_i(x)\xi^i$, then the coefficients $b_i(x)$ of the $1$-form $\beta$  is given by
			\begin{equation}\label{eqn2}
				b_i(x)=\frac{(1+\epsilon r^2) x^i }{(r^2-|x|^2)(1+\epsilon |x|^2)}, 
			\end{equation}
			Hence, \begin{equation}\label{eqn20}
				||\beta||^2_{\alpha}=a^{ij}b_ib_j=\frac{|x|^2(1+\epsilon r^2)^2}{r^2(1+\epsilon |x|^2)^2}  < 1.
			\end{equation}
			\noindent Moreover, it is easy to observe that  $\beta=df(x)$, where $f(x)=\frac{1}{2}\log\left(\frac{1+\epsilon |x|^2}{r^2-|x|^2} \right)$.
			Thus, $\mathcal{F}_{\epsilon}$  is a Randers metric with a closed $1$-form $\beta$.\\           
			\begin{rem} \begin{enumerate}[(i)]
					
					\item For $\epsilon=-1$, the Funk-Finsler metric $\mathcal{F}_{\epsilon}$ is defined on the disc $\mathbb{D}_E(r)$ with radius $r<1$, which can be realized as Funk-Finsler metric on Klein hyperbolic disc centered at the origin and hyperbolic radius $\tilde r=\frac{1}{2}\log \frac{1+r}{1-r}$. In particular if $r=\frac{e^2-1}{e^2+1}$, $\mathcal{F}_{\epsilon}$ can be realized as Funk-Finsler metric on Klein hyperbolic disc centered at origin and unit hyperbolic radius. 
					\item For $\epsilon=1$, the Funk-Finsler metric $\mathcal{F}_{\epsilon}$ is defined on the disc $\mathbb{D}_E(r)$ with radius $r<1$, which can be realized as Funk-Finsler metric on spherical geodesic disc of radius $\arctan r$.
				\end{enumerate}
			\end{rem}


			\section{Some curvatures of the Funk-Finsler structure on the constant curvature spaces}\label{sec4}
			In this section, we explicitly obtain the expressions for the $S$-curvature, the Riemann curvature, the Ricci curvature and the flag curvature of the Funk-Finsler metric $\mathcal{F}_\epsilon$ on a disc $\mathbb{D}_E(r)$ equipped with Riemannian metric of constant curvature $\epsilon=\{0,1,-1\}$. 
			
			\subsection{Spray coefficients and $S$-curvatures of the Funk-Finsler metric $\mathcal{F}_\epsilon$}\label{sb3}
			In this subsection,
			we recall the formula for $S$-curvature of 
			a general Randers metric $F=\alpha+\beta$, where $\alpha(x, \xi)=\sqrt{a_{ij}\xi^i\xi^j}$ and $\beta(x, \xi)=b_i(x)\xi^i$.\\
			Let  $\bar{\Gamma}^k_{ij}(x)$  denote the  Christoffel symbols of Riemannian metric $\alpha$. Then we have,
			\begin{equation}\label{eqnn 4.3.48}
				b_{i|j}:=\frac{\partial b_i}{\partial x^j}-b_k \bar{\Gamma}^k_{ij}.
			\end{equation}
			We introduce the following notations,
			\begin{equation}\label{eqnn 4.3.50}
				r_{ij}:=\frac{1}{2}\left( b_{i|j}+b_{j|i}\right),~~ s_{ij}:=\frac{1}{2}\left( b_{i|j}-b_{j|i}\right).
			\end{equation}
			\begin{equation}\label{eqnn 4.3.51}
				s^i_j:=a^{ih}s_{hj},~~ s_j:=b_i s^i_j=b^j s_{ij}, ~~r_j:=b^i r_{ij}, ~~b^j=a^{ij} b_i,
			\end{equation}
			\begin{equation}\label{eqnn 4.3.52}
				e_{ij}:=r_{ij}+b_is_j+b_js_i,
			\end{equation}
			\begin{equation*}e_{00}:=e_{ij}\xi^i\xi^j,~~ s_0:=s_i\xi^i ~~ \text{and}~~ s^i_0:=s_j^i\xi^j.\end{equation*}\\
			Now consider, 
			\begin{equation}\label{eqnA0}
				\rho:=\log \sqrt{\left( 1-||\beta||^2_\alpha\right)},~\mbox{and}~ \rho_0:=\rho_i\xi^i,\rho_i:=\rho_{x^i}(x).
			\end{equation}
			It is well known that the $S$-curvature of the Randers metric $F=\alpha +\beta $ is given by,
			\begin{equation}\label{eqnn 4.65}
				S=(n+1)\Big[\frac{e_{00}}{2F}-(s_0+\rho_0) \Big],
			\end{equation}
			see \cite[$\S 3.2$]{CXSZ} for more details.
			\vspace{.5cm}
			\begin{thm}\label{thm4.1}
				The  Funk-Finsler metric $\mathcal{F}_\epsilon$, , given in Theorem \ref{thm1.1},, on the disc $\mathbb{D}_E(r)$ equipped with Riemannian metric of constant curvature $\epsilon=\{0,1,-1\}$, is projectively flat and its $S$-curvature   is given by:
				\begin{equation*}
					\textbf{S}_\epsilon(x,\xi)= \frac{3(1+\epsilon r^2)\left[ (1+\epsilon |x|^2)|\xi|^2-2\langle x , \xi \rangle^2\right] }{2\mathcal{F}_\epsilon (r^2-|x|^2)(1+\epsilon |x|^2)^2}+\frac{3\langle x , \xi \rangle (1+\epsilon r^2)^2(1-\epsilon |x|^2)}{(r^2-|x|^2)(1-\epsilon^2 r^2|x|^2)(1+\epsilon |x|^2)},
				\end{equation*}
				where $(x,\xi) \in T\mathbb{D}_E(r)$.
			\end{thm}
			\begin{proof}
				From equation \eqref{eqnn 4.65}, to calculate the $S$-curvature of  Randers metric $\mathcal{F}_\epsilon$, given in Theorem \ref{thm1.1}, we proceed as follows. The Christoffel symbols $\bar{\Gamma}^k_{ij}(x)$ of Riemannian  metric $\alpha$  are given by
				\begin{eqnarray}
					\label{eqn0}  \bar{\Gamma}^k_{ij}(x):=\frac{x^i \delta_{kj}+x^j \delta_{ki}}{(r^2-|x|^2)},~~ \mbox{for all}~~ x\in \mathbb{D}_E(r).
				\end{eqnarray}
				Clearly,
				\begin{equation}\label{eqn01} 
					\bar{\Gamma}^k_{ij}(x)=\bar{\Gamma}^k_{ji}(x).
				\end{equation}
				Using expression for $b_i$ from \eqref{eqn2}, we get,
				\begin{equation}\label{eqnA1}
					\frac{\partial b_i}{\partial x^j}=\frac{\partial b_j}{\partial x^i}=\frac{(1+\epsilon r^2)}{(r^2-|x|^2)(1+\epsilon |x|^2)}\Bigg[ \delta_{ij}+\frac{2x^ix^j(1-\epsilon r^2+2\epsilon |x|^2)}{(r^2-|x|^2)(1+\epsilon |x|^2)} \Bigg].
				\end{equation}
				Substituting \eqref{eqn01} and \eqref{eqnA1}
				in \eqref{eqnn 4.3.48}, we obtain
				\begin{equation}\label{eqnA2}
					b_{i|j} =b_{j|i}=\frac{(1+\epsilon r^2)}{(r^2-|x|^2)(1+\epsilon |x|^2)}\Big[ \delta_{ij}-\frac{2\epsilon x^ix^j}{(1+\epsilon |x|^2)} \Big].
				\end{equation}
				Employing \eqref{eqnA2} in \eqref{eqnn 4.3.50}  yields,
				\begin{equation}\label{eqnA3}
					s_{ij}=0,~ r_{ij}=b_{i|j}=\frac{(1+\epsilon r^2)}{(r^2-|x|^2)(1+\epsilon |x|^2)}\Big[ \delta_{ij}-\frac{2\epsilon x^ix^j}{(1+\epsilon |x|^2)} \Big] ,~\forall i,j=1,2.
				\end{equation}
				Applying  \eqref{eqnA3} in \eqref{eqnn 4.3.51} and \eqref{eqnn 4.3.52}, we have
				\begin{equation}\label{eqnA4}
					s^i_j:=0,~ s_j:=0, ~ \mbox{and}~~ e_{ij}=r_{ij}=\frac{(1+\epsilon r^2)\left[ (1+\epsilon |x|^2)\delta_{ij}-2\epsilon x^ix^j\right] }{(r^2-|x|^2)(1+\epsilon |x|^2)^2}.
				\end{equation}
				Let $G^i=G^i(x,\xi)$ and $\bar{G}^i=\bar{G}^i(x,\xi)$ denote the spray coefficients of $\mathcal{F}_\epsilon$ and $\alpha$ respectively. Then  $G^i$ and $\bar{G}^i$ are related by (see \cite[\S 2.3, equation (2.19)]{CXSZ})
				\begin{equation}\label{eqnA5}
					G^i=\bar{G}^i+P \xi^i+Q^i,
				\end{equation}  
				where
				\begin{equation}\label{eqnA6}
					P:=\frac{e_{00}}{2\mathcal{F}_\epsilon}-s_0, ~~ ~~  Q^i:=\alpha s^i_0  ~~ \mbox{and}~~ \bar{G}^i=\frac{1}{2}\bar{\Gamma}^i_{jk}\xi^j\xi^k,
				\end{equation}
				and  where $e_{00}:=e_{ij}\xi^i\xi^j$, $s_0:=s_i\xi^i$ and $s^i_0:=s_j^i \xi^j$.\\
				Therefore from \eqref{eqn0} , \eqref{eqnA4} and \eqref{eqnA6} we find  that,
				\begin{equation}\label{eqnA7}
					P=\frac{e_{00}}{2\mathcal{F}_\epsilon}=\frac{r_{ij}\xi^i\xi^j}{2\mathcal{F}_\epsilon}= \frac{(1+\epsilon r^2)\left[ (1+\epsilon |x|^2)|\xi|^2-2\epsilon \langle x , \xi \rangle^2\right] }{2\mathcal{F}_\epsilon (r^2-|x|^2)(1+\epsilon |x|^2)^2},
				\end{equation}
				\begin{equation}\label{eqnA8}
					Q^i=0 ~ \mbox{and}~~ \bar{G}^i=\frac{\xi^i\langle x , \xi \rangle}{(r^2-|x|^2)},~  i=1,2.
				\end{equation}
				Employing \eqref{eqnA7}, \eqref{eqnA8} in  \eqref{eqnA5}, we see that the spray coefficients are given by
				\begin{equation}\label{eqnA9}
					G^i=\frac{\xi^i}{(r^2-|x|^2)}\left(\langle x , \xi \rangle+\frac{(1+\epsilon r^2)\left[ (1+\epsilon |x|^2)|\xi|^2-2\epsilon \langle x , \xi \rangle^2\right]}{2\mathcal{F}(1+\epsilon |x|^2)^2} \right). 
				\end{equation}
				Therefore, the  Funk-Finsler metric $\mathcal{F}_\epsilon$ is projectively flat (see Proposition \ref{ppn2.3}).\\
				From \eqref{eqn20} and \eqref{eqnA0} for $\mathcal{F}_\epsilon$ we have 
				\begin{equation}\label{rho}
					\rho=\frac{1}{2}\log\frac{\left(r^2- |x|^2 \right)\left(1-\epsilon r^2|x|^2 \right)}{r^2\left(1+\epsilon |x|^2 \right)^2}.
				\end{equation}
				Therefore,
				\begin{equation}\label{rho0}
					\rho_0=\rho_{x^i}\xi^i=-\frac{\left( 1+\epsilon r^2\right)^2\left(1-\epsilon |x|^2 \right)\langle x , \xi \rangle}{\left(r^2- |x|^2 \right)\left(1-\epsilon^2 r^2 |x|^2 \right)\left(1+\epsilon |x|^2 \right)}.
				\end{equation}
				Availing \eqref{rho0} and \eqref{eqnA7} in \eqref{eqnn 4.65} for $n=2$, we obtain $S$-curvature as:
				\begin{equation*}
					\textbf{S}_\epsilon(x,\xi)= \frac{3(1+\epsilon r^2)\left[ (1+\epsilon |x|^2)|\xi|^2-2\epsilon \langle x , \xi \rangle^2\right] }{2\mathcal{F}_\epsilon(x,\xi) (r^2-|x|^2)(1+\epsilon |x|^2)^2}+\frac{3\langle x , \xi \rangle (1+\epsilon r^2)^2(1-\epsilon |x|^2)}{(r^2-|x|^2)(1-\epsilon^2 r^2|x|^2)(1+\epsilon |x|^2)} .
				\end{equation*}
			\end{proof}
			
			\subsection{Proof of Theorem \ref{thm1.3}:}\label{sb30}
			\begin{proof}
				From Theorem \ref{thm4.1} we have
				\begin{eqnarray}
					\nonumber    S_\epsilon-\frac{3}{2}\mathcal{F}_{\epsilon}&=& \frac{3(1+\epsilon r^2)\left[ (1+\epsilon |x|^2)|\xi|^2-2\epsilon \langle x , \xi \rangle^2\right] }{2\mathcal{F}_\epsilon(x,\xi) (r^2-|x|^2)(1+\epsilon |x|^2)^2}\\
					\nonumber &&+\frac{3\langle x , \xi \rangle (1+\epsilon r^2)^2(1-\epsilon |x|^2)}{(r^2-|x|^2)(1-\epsilon^2 r^2|x|^2)(1+\epsilon |x|^2)}-\frac{3}{2}\mathcal{F}_\epsilon.
				\end{eqnarray}
				Therefore,
				\begin{eqnarray}
					\nonumber  S_\epsilon-\frac{3}{2}\mathcal{F}_{\epsilon} &=&\frac{3(1+\epsilon r^2)\left[ (1+\epsilon |x|^2)|\xi|^2-2\epsilon \langle x , \xi \rangle^2\right] }{2(r^2-|x|^2)(1+\epsilon |x|^2)^2\times \frac{\Big[\left(  1+\epsilon |x|^2\right)\sqrt{(r^2-|x|^2)|\xi|^2 +\langle x , \xi \rangle^2} +\left(  1+\epsilon r^2\right)\langle x , \xi \rangle\Big]}{(r^2-|x|^2)(1+\epsilon |x|^2)}}\\
					\nonumber &&+\frac{3\langle x , \xi \rangle (1+\epsilon r^2)^2(1-\epsilon |x|^2)}{(r^2-|x|^2)(1-\epsilon^2 r^2|x|^2)(1+\epsilon |x|^2)}\\
					\nonumber &&-\frac{3}{2}\frac{\Big[\left(  1+\epsilon |x|^2\right)\sqrt{(r^2-|x|^2)|\xi|^2 +\langle x , \xi \rangle^2} +\left(  1+\epsilon r^2\right)\langle x , \xi \rangle\Big]}{(r^2-|x|^2)(1+\epsilon |x|^2)}.
				\end{eqnarray}
				After simplification we obtain
				\begin{eqnarray}
					\nonumber  S_\epsilon-\frac{3}{2}\mathcal{F}_{\epsilon} &=&    \frac{\epsilon}{\left(  1+\epsilon |x|^2\right)\left(  1-\epsilon^2 r^2|x|^2\right)\Big[\left(  1+\epsilon |x|^2\right)\sqrt{(r^2-|x|^2)|\xi|^2 +\langle x , \xi \rangle^2} +\left(  1+\epsilon r^2\right)\langle x , \xi \rangle\Big]}\\
					\nonumber  &&\times \Bigg[\Big[\left(  1+\epsilon |x|^2\right)\sqrt{(r^2-|x|^2)|\xi|^2 +\langle x , \xi \rangle^2} +\left(  1+\epsilon r^2\right)\langle x , \xi \rangle\Big]^2 \\
					\label{S} &&~~~~~~~~~~~~~~~~~~+ \epsilon \left(r^2- |x|^2\right)\left(  1+\epsilon r^2\right)\left(  1+\epsilon |x|^2\right)\Big[\langle x , \xi \rangle^2  -|x|^2|\xi|^2\Big] \Bigg].
				\end{eqnarray}
				Clearly,  for the Euclidean Funk disc, $\epsilon=0$ and $|x| < r$. Therefore from \eqref{S} we have
				$$S_0-\frac{3}{2}\mathcal{F}_{0}=0.$$
				For the hyperbolic Klein Funk disc $\epsilon=-1$ and $|x|< r \leq 1$. Then from \eqref{S} we yield
				$$S_{-1}-\frac{3}{2}\mathcal{F}_{-1} <0,~~i.e.,~~S_{-1}< \frac{3}{2}\mathcal{F}_{-1}.$$
				For the spherical Funk disc $\epsilon=1$ and $|x|< r \leq 
				1$. Therefore from \eqref{S} we obtain
				$$S_1-\frac{3}{2}\mathcal{F}_{1} > 0 ~~i.e.,~~S_{1} > \frac{3}{2}\mathcal{F}_{1}.$$
			\end{proof}
			
		\begin{proposition}
			The Funk-Finsler metric $\mathcal{F}_\epsilon$ on the unit disc $\mathbb{D}_E(r)$, given in Theorem \ref{thm1.1}, is a non-Berwaldian Douglas metric.
		\end{proposition}
		
		\begin{proof}
			It is known that a Randers metric $F=\alpha+\beta$ is Berwald if and only if $\beta$ is parallel with respect to $\alpha$, i.e., $b_{i|j}=0$ (see Proposition \ref{ppn2.1}) and $F$ is Douglas if and only if $b_{i|j}=b_{j|i}$ (or $s_{ij}=0$, see Proposition  \ref{ppn2.2}).
			In view of \eqref{eqnA3} clearly  $\mathcal{F}_\epsilon$ is Douglas.\\
			If possible let $\mathcal{F}_\epsilon$ is Berwald, i.e., $b_{i|j}=0$. Since $|x|^2 < r \leq 1$, we have
			\begin{equation*}
				\delta_{ij}-\frac{2\epsilon x^i x^j}{1+\epsilon |x|^2}=0.
			\end{equation*}
			Put $i=j$ and summing over $i$, we get a contradiction, and hence $\mathcal{F}_\epsilon$ is a non-Berwaldian Douglas metric.
		\end{proof}
		\begin{thm}
			Let  $\mathcal{F}_\epsilon$ be the Funk metric on the disc $\mathbb{D}_E(r)$, given in Theorem \ref{thm1.1}. Then the Riemann curvature $R^i_k$, the Ricci curvature  $\textbf{Ric}$, and the flag curvature of  $\mathcal{F}_\epsilon$ are given by
			\begin{equation}
				R^i_k=- \left(  \delta^i_k \alpha^2-\alpha \alpha_k\xi^i \right)+\Bigg[ 3\left( \frac{\phi}{2\mathcal{F}_\epsilon}\right)^2  -\frac{\psi}{2\mathcal{F}_\epsilon}\Bigg]\left( \delta^i_k-\frac{(\mathcal{F}_\epsilon)_{\xi^k}}{\mathcal{F}_\epsilon}\xi^i\right) +\tau_k \xi^i,
			\end{equation}
			and 
			\begin{equation}\label{eqn 4.4.2A}
				\textbf{Ric}= \textbf{K} \mathcal{F}_\epsilon^2=\Bigg[3\left( \frac{\phi}{2\mathcal{F}_\epsilon}\right)^2  -\frac{\psi}{2\mathcal{F}_\epsilon} \Bigg]-\alpha^2,
			\end{equation} 
			where
			$$ \psi=\frac{-2(1+\epsilon r^2)\langle x , \xi \rangle}{(1+\epsilon |x|^2)^3(r^2-|x|^2)^2}\Bigg[ (1+\epsilon |x|^2)|\xi|^2(3\epsilon r^2-2\epsilon |x|^2+1)-2\epsilon  \langle x ,\xi \rangle^2(1-\epsilon |x|^2+2\epsilon r^2) \Bigg],$$

			$$\phi=\frac{(1+\epsilon r^2)\left[ (1+\epsilon |x|^2)|\xi|^2-2\epsilon \langle x , \xi \rangle^2\right] }{ (r^2-|x|^2)(1+\epsilon |x|^2)^2},$$
			and 
			\begin{equation*}
				\tau_k=\frac{(1+\epsilon r^2)(x^k |\xi|^2-\xi^k \langle x , \xi \rangle)}{\mathcal{F} (r^2-|x|^2)^2(1+\epsilon |x|^2)}.
			\end{equation*}
		\end{thm}
		\begin{proof} The Riemannian curvature of the Randers metric $F=\alpha+\beta$ with a closed $1$-form on an $n$-dimensional manifold is a given by (see \cite[$\S 5.2$, equation $(5.10)$ and equation $(5.12)$]{CXSZ})
			\begin{equation}\label{eqnn 4.3.70}
				R^i_k=\overline{R^i_k}+\Bigg[ 3\left( \frac{\phi}{2F}\right)^2  -\frac{\psi}{2F}\Bigg]\left( \delta^i_k-\frac{F_{\xi^k}}{F}\xi^i\right) +\tau_k \xi^i,
			\end{equation}
			where
			\begin{equation}\label{eqnn 4.3.72}
				\phi:=b_{i|j}\xi^i\xi^j,~~~\psi:=b_{i|j|k}\xi^i\xi^j\xi^k,~~~\tau_k:=\frac{1}{F}\left( b_{i|j|k}-b_{i|k|j}\right) \xi^i\xi^j,
			\end{equation}
			and 
			\begin{equation}\label{eqnn 4.3.720}
				b_{i|j|k}=\frac{\partial b_{i|j}}{\partial x^k}-b_{i|m}\bar{\Gamma}^m_{jk}-b_{j|m}\bar{\Gamma}^m_{ik}.   
			\end{equation}
			Here $\overline{R^i_k}$   denotes the Riemann curvature of the Riemannian metric $\alpha$.\\
			Further,  the Ricci curvature of the Randers metric is given by
			\begin{equation}\label{eqnn 4.3.71}
				\textbf{Ric} =\overline{\textbf{Ric}} +(n-1)\Bigg[3\left( \frac{\phi}{2F}\right)^2  -\frac{\psi}{2F} \Bigg],
			\end{equation} 
			where $\overline{\textbf{Ric}}$   denotes the Ricci curvature of the Riemannian metric $\alpha$.\\
			It is well known that the Gaussian curvature of the Riemannian metric $\alpha=\frac{\sqrt{\left(r^2-|x|^2 \right) |\xi|^2+ \langle x, \xi \rangle ^2 }}{r^2-|x|^2}$   is $-1$. Therefore, from \eqref{eqn2.1.12} the Riemann curvature of the  metric is given by 
			\begin{equation}\label{eqnn 4.3.710}
				\overline{R^i_k}=- \left(  \delta^i_k \alpha^2-\alpha \alpha_k\xi^i \right),~ \alpha_k:=\frac{\partial \alpha}{\partial \xi ^k}.
			\end{equation} 
			Considering \eqref{eqnn 4.3.70} with \eqref{eqnn 4.3.710}, we get the desired expression for the Riemann curvature.\\
			In \eqref{eqnn 4.3.710}, use $i=k$ and then sum over $i$ for  $i =1,2$, we have,
			\begin{equation}\label{eqnn 4.3.711}
				\overline{R^i_i}=-\alpha^2.
			\end{equation} 
			In view of \eqref{eqnA3}, \eqref{eqnn 4.3.720} for the Funk metric $\mathcal{F}_\epsilon$, given in Theorem \ref{thm1.1}, the functions $\phi$ and $\psi$ can be explicitly calculated as follows:
			\begin{eqnarray}
				\label{eqnn 4.3.74} \phi:=b_{i|j}\xi^i\xi^j=\frac{(1+\epsilon r^2)\left[ (1+\epsilon |x|^2)|\xi|^2-2\epsilon \langle x , \xi \rangle^2\right] }{ (r^2-|x|^2)(1+\epsilon |x|^2)^2},
			\end{eqnarray}
			and
			\begin{eqnarray}
				\nonumber \psi&=&b_{i|j|k}\xi^i\xi^j\xi^k=\left(\frac{\partial b_{i|j}}{\partial x^k}-b_{i|m}\bar{\Gamma}^m_{jk}-b_{j|m}\bar{\Gamma}^m_{ik}\right)\xi^i\xi^j\xi^k\\
				\nonumber &=&\frac{- 2(1+\epsilon r^2)\langle x , \xi \rangle}{(1+\epsilon |x|^2)^3(r^2-|x|^2)^2}\Bigg[ (1+|x|^2)|\xi|^2(3\epsilon r^2-2\epsilon |x|^2+1)\\
				\label{eqnn 4.3.75} &~&~~~~~~~~~~~~~~~~~~~~~~~~~~~~~~~~~~~~~~~~~~~~~-2 \epsilon \langle x ,\xi \rangle^2(1-\epsilon |x|^2+2\epsilon r^2) \Bigg].
			\end{eqnarray}
			Using \eqref{eqnA3} and \eqref{eqnn 4.3.720} in \eqref{eqnn 4.3.72}, we obtain
			\begin{eqnarray}
				\label{eqnn 4.3.76} \tau_k=\frac{(1+\epsilon r^2)(x^k |\xi|^2-\xi^k \langle x , \xi \rangle)}{\mathcal{F}_\epsilon (r^2-|x|^2)^2(1+\epsilon |x|^2)}.
			\end{eqnarray}
			Clearly, 
			\begin{equation}\label{eqnn 4.3.77}
				\tau_k\xi^k=0.
			\end{equation}
			Consequently, substituting \eqref{eqnn 4.3.711},  \eqref{eqnn 4.3.77} in \eqref{eqnn 4.3.70},
			putting, $i=k$ and using over $i$ for  $i =1,2$, we obtain:
			\begin{equation}\label{eqnn 4.3.712}
				R^i_i=-\alpha^2+\Bigg[ 3\left( \frac{\phi}{2\mathcal{F}_\epsilon}\right)^2  -\frac{\psi}{2\mathcal{F}_\epsilon}\Bigg],
			\end{equation}
			where $\phi$ , $\psi$ respectively, is given by \eqref{eqnn 4.3.74}, \eqref{eqnn 4.3.75}. \\
			Thus the Ricci curvature for  $\mathcal{F}_\epsilon$ is:
			\begin{eqnarray}
				\nonumber \textbf{Ric}=R^i_i =\Bigg[3\left( \frac{\phi}{2\mathcal{F}_\epsilon}\right)^2  -\frac{\psi}{2\mathcal{F}_\epsilon} \Bigg]-\alpha^2.
			\end{eqnarray}
			For dimension $2$,  $\textbf{K}=\frac{\textbf{Ric}}{\mathcal{F}_\epsilon^2}$, we have the desired result.
		\end{proof}
		
		\subsection{Proof of Theorem \ref{thm1.2}}\label{sb4}
		\begin{proof} 
			From \eqref{eqn 4.4.2A}, we have $$\textbf{K}=\frac{\textbf{Ric}}{\mathcal{F}_\epsilon^2}=\frac{\Big[3\left( \frac{\phi}{2\mathcal{F}_\epsilon}\right)^2  -\frac{\psi}{2\mathcal{F}_\epsilon} \Big]-\alpha^2}{\mathcal{F}_\epsilon^2}.$$
			Employing  $\mathcal{F}_\epsilon$, given in Theorem \ref{thm1.1}, \eqref{eqnn 4.3.74} and \eqref{eqnn 4.3.75} we obtain 
			\begin{eqnarray}
				\nonumber   \textbf{K}&=&\frac{3}{4}\frac{\phi^2}{\mathcal{F}_\epsilon^4}-\frac{\psi}{2\mathcal{F}_\epsilon^3}-\frac{\alpha^2}{\mathcal{F}_\epsilon^2}\\
				\nonumber  &=&\frac{3}{4}\frac{(r^2-|x|^2)^2(1+\epsilon r^2)^2\left[ (1+\epsilon |x|^2)|\xi|^2-2\epsilon \langle x , \xi \rangle^2\right]^2 }{\Big[\left(  1+\epsilon |x|^2\right)\sqrt{(r^2-|x|^2)|\xi|^2 +\langle x , \xi \rangle^2} +\left(  1+\epsilon r^2\right)\langle x , \xi \rangle\Big]^4}\\
				\nonumber &&+ \frac{\left(  1+\epsilon r^2\right) \left(r^2-|x|^2\right)\langle x , \xi \rangle \Big[(1+|x|^2)|\xi|^2(3\epsilon r^2-2\epsilon |x|^2+1)-2 \epsilon \langle x ,\xi \rangle^2(1-\epsilon |x|^2+2\epsilon r^2)\Big]}{\Big[\left(  1+\epsilon |x|^2\right)\sqrt{(r^2-|x|^2)|\xi|^2 +\langle x , \xi \rangle^2} +\left(  1+\epsilon r^2\right)\langle x , \xi \rangle\Big]^3}\\
				\nonumber && - \frac{\left(  1+\epsilon |x|^2\right)^2 \left((r^2-|x|^2)|\xi|^2 +\langle x , \xi \rangle^2 \right)}{\Big[\left(  1+\epsilon |x|^2\right)\sqrt{(r^2-|x|^2)|\xi|^2 +\langle x , \xi \rangle^2} +\left(  1+\epsilon r^2\right)\langle x , \xi \rangle\Big]^2}.
			\end{eqnarray}
			Further,
			
			\begin{eqnarray}
				\nonumber \textbf{K}+\frac{1}{4} 
				&=&\frac{1}{4\Big[\left(  1+\epsilon |x|^2\right)\sqrt{(r^2-|x|^2)|\xi|^2 +\langle x , \xi \rangle^2} +\left(  1+\epsilon r^2\right)\langle x , \xi \rangle\Big]^4} \times\\
				\nonumber &&\Bigg[3(r^2-|x|^2)^2(1+\epsilon r^2)^2\left[ (1+\epsilon |x|^2)|\xi|^2-2\epsilon \langle x , \xi \rangle^2\right]^2\\
				\nonumber  &&+ \left(  1+\epsilon r^2\right) \left(r^2-|x|^2\right)\langle x , \xi \rangle \Big[ \left(  1+\epsilon |x|^2\right)\sqrt{(r^2-|x|^2)|\xi|^2 +\langle x , \xi \rangle^2} +\left(  1+\epsilon r^2\right)\langle x , \xi \rangle\Big] \\
				\nonumber &&~~~~~~~~~~~~~~~~~~~~\times \Big[(1+|x|^2)|\xi|^2(3\epsilon r^2-2\epsilon |x|^2+1)-2 \epsilon \langle x ,\xi \rangle^2(1-\epsilon |x|^2+2\epsilon r^2)\Big]
				\\
				\nonumber &&-\Big[ \left(  1+\epsilon |x|^2\right)^2 \left((r^2-|x|^2)|\xi|^2 +\langle x , \xi \rangle^2 \right)\Big]\\
				\nonumber &&~~~~~~~~~~~~~~~~~~~~~\times \Big[\left(  1+\epsilon |x|^2\right)\sqrt{(r^2-|x|^2)|\xi|^2 +\langle x , \xi \rangle^2} +\left(  1+\epsilon r^2\right)\langle x , \xi \rangle\Big]^2\\
				\nonumber&&+ \Big[\left(  1+\epsilon |x|^2\right)\sqrt{(r^2-|x|^2)|\xi|^2 +\langle x , \xi \rangle^2} +\left(  1+\epsilon r^2\right)\langle x , \xi \rangle\Big]^4 \Bigg].
			\end{eqnarray}
			After simplifying we get
			\begin{eqnarray}
				\nonumber  \textbf{K}+\frac{1}{4} &=& \frac{3 \epsilon \left(r^2-|x|^2\right)^2\Big[ 1+\epsilon \left(|x|^2|\xi|^2 - \langle x , \xi \rangle^2 \right)\Big]}{4 \Big[\left(  1+\epsilon |x|^2\right)\sqrt{(r^2-|x|^2)|\xi|^2 +\langle x , \xi \rangle^2} +\left(  1+\epsilon r^2\right)\langle x , \xi \rangle\Big]^4}  \\
				\nonumber &&~~\times \Bigg[ \left(  1+\epsilon r^2\right)\left(  1+\epsilon |x|^2\right)\left(\sqrt{(r^2-|x|^2)|\xi|^2 +\langle x , \xi \rangle^2} +\langle x , \xi \rangle \right)^2\\
				\label{K}&&~~~~~~~+\left(\left(  1+\epsilon |x|^2\right)\sqrt{(r^2-|x|^2)|\xi|^2 +\langle x , \xi \rangle^2} +\left(  1+\epsilon r^2\right)\langle x , \xi \rangle \right)^2 \Bigg].
			\end{eqnarray}
			Clearly,  for the Euclidean Funk disc $\epsilon=0$ and $|x| < r$. Therefore from \eqref{K} we have
			\begin{equation*}
				\textbf{K} = -\frac{1}{4}.
			\end{equation*}
			For the hyperbolic Funk disc $\epsilon=-1$ and $|x|< r \leq 1$. Then from \eqref{K} we yield
			\begin{eqnarray}
				\nonumber   \textbf{K} +\frac{1}{4} < 0,~~~ \mbox{i.e.,} ~~~~\textbf{K} < -\frac{1}{4}.
			\end{eqnarray}
			Further, for the spherical Funk disc $\epsilon=1$ and $|x|< r \leq 
			1$. Therefore from \eqref{K} we obtain
			\begin{eqnarray}
				\nonumber  \textbf{K} +\frac{1}{4} > 0,~~~ \mbox{i.e.,} ~~~~   \textbf{K} > -\frac{1}{4}.
			\end{eqnarray}
		\end{proof}
		\section{Zermelo Navigation}\label{sec5}
		It is well known that any Randers metric on a manifold $M$ has a Zermelo navigation representation. For instant, if $F=\alpha+\beta$ is given the Randers metric with $\alpha=\sqrt{a_{ij}(x)\xi^i\xi^j}$ and differential $1$-form $\beta=b_i(x)\xi^i$, satisfying $||\beta||^2_\alpha=a^{ij}b_ib_j < 1$. Then the Zermelo Navigation for this Randers metric is the triple $(M,h,W)$, where  $h=\sqrt{h_{ij}\xi^i \xi^j}$ with
		\begin{equation*}
			h_{ij}= c (a_{ij}-b_ib_j), ~~ W^i=-\frac{b^i}{c},~ b^i=a^{ij}b_j~ \mbox{and}~ c=1-||\beta||^2_\alpha.
		\end{equation*}
		Moreover, $||W||_h=||\beta||_{\alpha}$.\\
		Also given the Zermelo data, we can get back the Randers metric. 
		And this 1-1 correspondence is useful in finding the geodesics of the Randers metric. 
		See for more details  \cite[ Example $1.4.3$]{SSZ}.\\\\
		In this subsection, we obtain the Zermelo data for the Funk-Finsler metric $\mathcal{F}_\epsilon$, which is a trivially a Randers metric.
		We have,
		\begin{equation*}
			\mathcal{F}_\epsilon(x, \xi) = \frac{\sqrt{\left(r^2-|x|^2 \right) |\xi|^2+ \langle x, \xi \rangle ^2 }}{r^2-|x|^2}+\frac{(1+\epsilon r^2) \langle x,\xi \rangle}{(r^2-|x|^2)(1+\epsilon|x|^2)}.
		\end{equation*}
		We need to find  $h_{ij}, W^i$ defined above. From 
		\eqref{eqn20}  we have,  
		\begin{equation}
			c=  1- ||\beta_F||^2_{\alpha_F}=1-\frac{|x|^2(1+\epsilon r^2)^2}{r^2(1+\epsilon |x|^2)^2} =\frac{(r^2-|x|^2)(1-\epsilon^2 r^2|x|^2)}{r^2(1+\epsilon |x|^2)^2}.  
		\end{equation}
		Employing \eqref{eqn2.5.118} and \eqref{eqn2}, we see  
		\begin{equation}\label{eqn 4.100A}
			h_{ij} =\frac{(1-\epsilon^2 r^2|x|^2)}{r^2(1+\epsilon |x|^2)^4} \Bigg[ \delta_{ij} (1+\epsilon |x|^2)^2-\epsilon x^ix^j ( 2+\epsilon r^2 +\epsilon |x|^2) \Bigg], 
		\end{equation} 
		Clearly,
		\begin{equation}\label{eqn 4.101A}
			W=\left( W^i \right)=\left( \frac{(1+\epsilon r^2)(1+\epsilon |x|^2)}{(1-\epsilon^2 r^2|x|^2)}x^i  \right).
		\end{equation}
		Also, 
		$$||W||^2_h=||\beta_F||^2_{\alpha_F}=\frac{|x|^2(1+\epsilon r^2)^2}{r^2(1+\epsilon |x|^2)^2}. $$
		Thus we have 
		\begin{proposition}
			The Zermelo Navigation data for the Funk-Finsler structure $\mathcal{F}_\epsilon$ on the disc is given by $\left( \mathbb{D}_E(r), h, W\right)$ where the components of the Riemannian metric $h$ is given by \eqref{eqn 4.100A} and that of the vector field by \eqref{eqn 4.101A}.
		\end{proposition}

		\section{Appendix}\label{sec6}
		In view of \eqref{eqn2.5.11811} and the fact that $(\mathbb{R}^2,g)$ is projectively flat, i.e., its geodesics are line segments. Consider $x,y \in \mathbb{D}_E(r)\subset \mathbb{R}^2$ and $\mathfrak{a} \in \partial \mathbb{D}_E(r)$ such that $x,y,\mathfrak{a}$ are collinear points, let $\gamma(t)=x+t\xi$ be geodesic, where $\xi$ is a Euclidean unit vector along $\vec{x \mathfrak{a}}$, that is, $\xi= \frac{y-x}{|y-x|}=\frac{\mathfrak{a}-x}{|\mathfrak{a}-x|}$. Clearly, $\gamma(0)=x$,~$\gamma(|y-x|)=y$ and $\gamma(|\mathfrak{a}-x|)=\mathfrak{a}$.\\
		
		Therefore,
		\begin{eqnarray}
			\nonumber   g(\gamma(t),\dot{\gamma}(t))&=&\frac{\sqrt{(1+|\gamma(t)|^2)|\dot{\gamma}(t)|^2-\langle \gamma(t), \dot{\gamma}(t) \rangle ^2}}{(1+|\gamma(t)|^2)}=\frac{k}{k^2+(t+\langle x , \xi \rangle)^2},
		\end{eqnarray}
		where $k=\sqrt{1+|x|^2-\langle x , \xi \rangle^2}$. Then
		\begin{eqnarray}
			\nonumber d_s(\gamma(t_1),\gamma(t_2))=\int_{t_1}^{t_2} g(\gamma(t),\dot{\gamma}(t))  d t = \int_{t_1}^{t_2}  \frac{k}{k^2+(t+\langle x , \xi \rangle)^2} d t= \Bigg[ \tan ^{-1} \frac{t+\langle x , \xi \rangle}{k}\Bigg ] _{t_1}^{t_2}.
		\end{eqnarray}
		\begin{eqnarray}
			d_s(\gamma(t_1),\gamma(t_2))=  \tan ^{-1} \frac{(t_2-t_1)k}{k^2+(t_1+\langle x , \xi \rangle)(t_2+\langle x , \xi \rangle)}.
		\end{eqnarray}
		Therefore,
		
		\begin{eqnarray}
			\nonumber   \tan d_s(x,\mathfrak{a}) =\frac{k |\mathfrak{a}-x|}{1+x^2+|\mathfrak{a}-x| \langle x , \xi \rangle},
		\end{eqnarray}
		and
		\begin{eqnarray}
			\nonumber \tan d_s(y,\mathfrak{a}) &=&   \frac{k (|\mathfrak{a}-x|-|y-x|)}{k^2+(|y-x|+\langle x , \xi \rangle)(|\mathfrak{a}-x|+\langle x , \xi \rangle)}= \frac{k |\mathfrak{a}-y|}{k^2+A^2-A|\mathfrak{a}-y|},
		\end{eqnarray}
		where $A=|\mathfrak{a}-x|+\langle x , \xi \rangle$.\\
		Hence, by simple trigonometric computation we obtain
		\begin{eqnarray}
			\label{xa}\sin d_s(x,\mathfrak{a}) &=&  \frac{k |\mathfrak{a}-x|}{\sqrt{k^2|\mathfrak{a}-x|^2+(1+x^2+|\mathfrak{a}-x| \langle x , \xi \rangle)^2}}=\frac{k |\mathfrak{a}-x|}{\sqrt{1+|x|^2}\sqrt{\lambda}},
		\end{eqnarray}
		
		and
		
		\begin{equation}\label{ya}
			\sin d_s(y,\mathfrak{a})=\frac{k |\mathfrak{a}-y|}{\sqrt{k^2|\mathfrak{a}-y|^2+(k^2+A^2-A|\mathfrak{a}-y|)^2}}= \frac{k |\mathfrak{a}-y|}{\sqrt{1+|y|^2}\sqrt{\lambda}},
		\end{equation}
		where $\lambda =|\mathfrak{a}-x|^2+1+|x|^2+2|\mathfrak{a}-x|\langle x , \xi \rangle$.\\
		Thus, from \eqref{xa} and \eqref{ya} we yield
		\begin{equation}
			\frac{\sin d_s(x,\mathfrak{a})}{\sin d_s(y,\mathfrak{a})}=\frac{|\mathfrak{a}-x|}{|\mathfrak{a}-y|}\frac{\sqrt{1+|y|^2}}{\sqrt{1+|x|^2}}.
		\end{equation}
		Therefore, from the above equation we obtain
		$$d(x,y)  = \log \frac{|x-\mathfrak{a}|}{|y-\mathfrak{a}|}  +\frac{1}{2}\log \frac{1+|y|^2}{1+|x|^2}.$$
		
		\vspace{0.8 cm}
		
		\noindent \textbf{Funding:}  The corresponding author, Bankteshwar Tiwari, is supported by ``Incentive grant" under the IoE scheme of Banaras Hindu University, Varanasi(India).

	\end{document}